%% file: product21.tex
\documentclass{amsart}

\newtheorem{theorem}{Theorem}[section]

\newtheorem{prop}[theorem]{Proposition}

\newtheorem{remark}[theorem]{Remark}

\numberwithin{equation}{section}
\newcommand{\BR}{\mathbb{R}}

\newcommand{\BZ}{\mathbb{Z}}
\newcommand{\BQ}{\mathbb{Q}}

\usepackage{calc,amssymb,amsmath,amsthm,amsfonts,ifthen,graphics,enumerate,color,epsfig}

\begin{document}

\title{Sigma Invariants of Direct Products of Groups}

\author{Robert Bieri and Ross Geoghegan}

\address{Department of Mathematics, 
Johann Wolfgang Goethe-Universit\"at Frankfurt, 
D-60054 Frankfurt am Main, 
Germany \newline
\newline
Department of Mathematical Sciences, 
Binghamton University (SUNY), 
Binghamton, NY 13902-6000, USA}
\email{bieri@math.uni-frankfurt.de,
ross@math.binghamton.edu}

\thanks{This work was partially supported by a grant from the Deutsche Forschungsgemeinschaft.}


\date{May, 2008}


\begin{abstract} The Product Conjecture for the homological
Bieri-Neumann-Strebel-Renz invariants is proved over a field.  Under
certain hypotheses the Product Conjecture is shown to also hold over
$\BZ$, even though D. Sch\"{u}tz has recently shown that the Conjecture
is false in general over $\BZ$.  Our version over $\BZ$ is applied in
a joint paper with D. Kochloukova \cite{BGK} to show that for all $n$
Thompson's group $F$ contains subgroups of type $F_n$ which are not of
type $FP_{n+1}$.  \end{abstract}

\maketitle

\section{Introduction}\label{Intro}

Let $G$ be a group. A (real additive) {\it character} on $G$ is a homomorphism $\chi :G\to \BR$ from $G$ to the additive group of real numbers. Two non-zero characters $\chi  ,\chi' :
G \to \BR$ are {\it equivalent} if they differ by a positive multiple,
i.e. $\chi' = r\chi$ for some $r > 0$. The equivalence class of $\chi$
is denoted by $[\chi]$.  The equivalence class of a non-zero character
$\chi $ should be thought of as the straight open ray from $0$ through $\chi $ in
the real vector space Hom$(G,\BR)$ of all characters.  The dimension of
this vector space is the torsion-free rank of the abelianization $G/G'$ of $G$. Thus, when $G/G'$ is finitely generated, the set of equivalence classes of non-zero characters, denoted by $S(G)$, is a geometric sphere on which we can do spherical geometry.

Our ground ring $R$ is assumed to be a domain, i.e., a commutative
ring with $1\neq 0$ which is without zero divisors. When we consider
$R$ as an $RG$-module, we always refer to the trivial action of $G$ on $R$.
Each $[\chi]$ defines a submonoid $G_{\chi}:=\{g\in G\: |\:\chi (g)\geq 0\}$,
and $R$ will also be considered as a trivial $RG_{\chi }$-module.   In the
paper \cite{BieriRenz} of  Bieri and  Renz\footnote{In \cite{BieriRenz}
things are done over the ground ring $\BZ$, but everything in that paper
goes through over $R$. In our application to Thompson's Group in \cite{BGK} we
will need to compare the situation over $\BZ$ with that over $\BQ$.} the
$\Sigma$-{\it invariants} (or {\it geometric invariants}) are defined for
each integer $n\geq 0$ by 
$$\Sigma ^{n}(G;R):=\{[\chi ]\in S(G)\: |\:R\text{ is of type } FP_{n}\text{ over } RG_{\chi }\}$$ 
The basic reference for
these is \cite{BieriRenz}. A less economical but more intuitive definition is given in Section \ref{Sigma}. While it is clear that $\Sigma ^{0}(G;R)=S(G)$
the case $n=0$ will play a role in what follows.

There are also homotopical versions of these invariants,
denoted $\Sigma ^{n}(G)$, which were introduced in
\cite[Remark~6.1]{BieriRenz}\footnote{In \cite{BieriRenz} these are
denoted by $^*\Sigma ^{n}(G)$. Again, we have $\Sigma ^{0}(G)=S(G)$. The
case $n=1$ is a recasting of the Bieri-Neumann-Strebel Invariant
introduced earlier in \cite{BNS}.}. The relationship between
$\Sigma^{n}(G;R)$ and $\Sigma ^{n}(G)$ is the usual relationship in
topology between homology with $R$-coefficients and homotopy; see
Theorem \ref{3.3}.

This paper is about the behavior of the $\Sigma$-invariants with respect
to direct products of groups. Consider two groups $G$ and $H$. The vector
spaces Hom$(G\times H,\BR)$ and Hom$(G,\BR)\oplus \text{Hom}(H,\BR)$
are identified in the usual way.  This embeds the spheres $S(G)$
and $S(H)$ canonically as subspheres of $S(G\times H)$ so that when
$\chi \in \text{Hom}(G,\BR )$ and $\chi  '\in \text{Hom}(H,\BR )$ the
notations $\chi +\chi '$ and $(\chi ,\chi ')$ both describe the character
$(g,h)\mapsto \chi (g)+\chi '(h)$. Thus, as the non-zero characters $\chi
$  and $\chi '$ vary in their respective rays $[\chi ]$ and $[\chi ']$,
the point $[\chi +\chi ']\in S(G\times H)$ varies among the points of the
open spherical geodesic whose end points are $[\chi]$ and $[\chi ']$.
With the usual interpretation of  $S(G\times H)$ as the join of $S(G)$
and  $S(H)$, this spherical geodesic is the join-segment from $[\chi ]$
to $[\chi ']$.  Thus, when $P\subseteq S(G)$ and $Q\subseteq S(H)$ their
{\it join} is $$P*Q:=\{[\chi +\chi ']\: |\:[\chi]\in P, [\chi ']\in
Q\}\cup P\cup Q$$ It is sometimes convenient to extend the notations
$\chi +\chi '$ and $(\chi ,\chi ')$ to include the possibility $\chi =0$
or $\chi '=0$ thus collecting the two endpoints of the geodesic in the
same notation. (The reason is that when $\chi '=0$ and $\chi \neq 0$,
or vice versa, $\chi +\chi '$ is a non-zero character on $G\times H$.)

Recall that the group $G$ {\it has type} $FP_n(R)$ if there is a free
$RG$-resolution of $R$ which is finitely generated in dimensions $\leq
n$; when $R=\BZ$ one simply says that $G$ {\it has type }$ FP_n $.
When we discuss $\Sigma^n(G;R)$, {\it we will always assume} that
the group in question has type $FP_{n}(R)$.  Similarly, $\Sigma^{n}(G)$ is only defined when $G$ has the topological finiteness
type\footnote{$G$ has {\it type} $F_n$ if there is a $K(G,1)$-complex
with finite $n$-skeleton.  For details on finiteness properties of groups
see for example \cite{Rossbook}.} $F_n$.

The question of a formula for the homotopical $\Sigma $-invariants of
direct products of groups is reduced to the corresponding question for
the homological invariants by the following:

\begin{theorem} $$ \Sigma^n(G\times H)= (\Sigma^n(G\times H;\BZ)-(S(G)\cup
S(H)))\cup (\Sigma^n(G)\cup \Sigma^n(H)) $$ \end{theorem}

This first appeared in \cite{MMVW2}; see \cite[Proposition~4.6]{Bieri}
for a proof.

We denote the complement of any subset $A$ of a sphere by $A^c$.
The {\bf Direct Product Formula}
(whether true or false - that is the subject of this paper) reads as follows:

$$ \Sigma^n(G\times H;R)^c =\bigcup^n_{p=0} \Sigma^p(G;R)^c *
\Sigma^{n-p} (H;R)^c$$
								
The $\subseteq$ inclusion of this statement is a theorem due to H. Meinert:

\begin{theorem}(Meinert's Inequality)\label{Meinert}
$$ \Sigma^n(G\times H; R)^c \subseteq \bigcup^n_{p=0} \Sigma^p(G;R)^c *
\Sigma^{n-p} (H;R)^c$$ and $$ \Sigma^n(G\times H)^c \subseteq \bigcup^n_{p=0} \Sigma^p(G)^c *\Sigma^{n-p} (H)^c$$
\end{theorem}

Meinert did not publish this, but a proof can be found in
\cite[Section~9]{Gehrke2}. The paper \cite{Bieri} also contains a proof of
the homotopy version\footnote{That our statement of Meinert's Inequality
is equivalent to the statement in \cite{Gehrke2} requires a little work.
The case where one of the characters is zero is covered by our Proposition
\ref{extreme}.  The other case is a straightforward exercise.}.

In this paper we consider the $\supseteq$ inclusion of the Direct
Product Formula.  Here, caution is needed as there are counterexamples.  A counterexample to the homotopy version was given by Meier, Meinert and vanWyk in \cite[Section~6]{MMVW}.
Recently, a counterexample to the homological version has been given by D. Sch\"{u}tz in \cite{Dirk} for the case where $R=\BZ$.  This involves the product of two right-angled Artin groups.  (The article \cite{Bieri} contains the {\it incorrect} statement - for which R.B. owes an apology - that the computation of $\Sigma ^n(G,\BZ)$ for right angled Artin groups $G$ given in \cite{MMVW} -- see also \cite{BuxG} -- establishes the $\supseteq$ direction of the Direct Product Formula for those groups when $R=\BZ$.)

In this paper we prove:

\begin{theorem}\label{Thm1} When $R$ is a field, the Direct Product Formula is true.
\end{theorem}

As a corollary we have:

\begin{theorem}\label{Thm2} When $\Sigma ^{p}(G;\BZ)=\Sigma ^{p}(G;\BQ)$ and 
$\Sigma ^{p}(H;\BZ)=\Sigma ^{p}(H;\BQ)$ for all $p\leq n$ then the $\BZ$-version of the
Direct Product Formula is true; i.e., $$\Sigma ^{n}(G\times H; \BZ)^{c}
=\bigcup^n_{p=0}\Sigma ^p(G;\BZ)^{c} * \Sigma^{n-p}(H;\BZ )^{c}$$ 
\end{theorem}

\begin{proof} This follows from Meinert's Inequality (Theorem
\ref{Meinert}) together with the fact that one always has:  $\Sigma
^{p}(.;\BZ)\subseteq \Sigma ^{p}(.;\BQ)$.  
\end{proof}

Theorem \ref{Thm2} is applied in \cite{BGK} to get information about
subgroups of Thompson's Group $F$.  There the $\Sigma $-invariants
of $F$ are computed, and one has $\Sigma ^{p}(F;\BZ)=\Sigma ^{p}(F;\BQ)$ for all $p$, so that, by Theorem \ref{Thm2}, the $\BZ$ version of the Direct Product Formula is true with $G=H=F$. A consequence
is that, for all $n$, $F$ contains subgroups of type\footnote{In general the $F_n$
property is stronger than the $FP_n$  property.} $F_n$  which
are not of type $FP_{n+1}$

A corollary of the proof of Theorem \ref{Thm1}, also proved by Sch\"{u}tz in \cite{Dirk}, is:

\begin{theorem}\label{Thm3} The Direct Product Formula is true when $R=\BZ$ provided $n\leq 3$.  
\end{theorem}

This is explained in Remark \ref{rational}.

We wish to acknowledge the roots of the present paper. Computing $ \Sigma^n(G\times H; \BZ)$ was a much discussed theme in Frankfurt around 1990, and in this paper we take full advantage of those discussions, more than we are able to track down in the literature. We can certainly refer to Holger Meinert's Diplome Thesis \cite{Meinert} which contains the $\BZ$-version of his Inequality, stated above in Theorem \ref{Meinert}.  We can also refer to Ralf Gehrke's doctoral thesis \cite{Gehrke} and \cite{Gehrke2} which implies equality in the $\BZ$-version in some special cases.  The survey article \cite{Bieri} contains more details, including a proof of Gehrke's result.

\section{Valuations extending a character}\label{Basic}

Let $\chi :G\to \BR$ be a character on $G$.  If $M$ is an $RG$-module, a {\it valuation on} $M$ {\it extending} $\chi$ is a function $v : M\to \BR \cup\{\infty\}$ satisfying the axioms

\begin{equation*}\begin{aligned}
 v(m+m') &\geq \min\{v(m),v(m')\}\\
 v(gm) &= \chi(g) + v(m)\\
 v(rm) &= v(m) \ \text{ when }r\ \text{ is a unit in }\ R\\
 v(0) &= \infty 
\end{aligned}\end{equation*}
for all $m,m' \in M$, $g\in G$, $r\in R$.

The {\it support} of $c\in F_i$ is the subset supp$(c)$ of the $R$-basis $G\mathcal{X}_i$
consisting of those members which appear with non-zero coefficients in the unique expansion of $c$.

The most important example of a valuation involves extension via supports: Given a free $RG$-module $F$ and an $RG$-basis $\mathcal{X}$ for $F$, every function $v : \mathcal{X} \to \BR$ extends to a unique function $v : G\mathcal{X} \to \BR \cup
\{\infty\}$ by $v(gx) := \chi(g) + v(x)$, and then to a valuation on $F$
by defining, for each non-zero $c\in F$, $v(c) =v(\text {\rm supp}(c))$.

Let $\mathbf F \twoheadrightarrow R$ be a free resolution of the trivial $RG$-module $R$.  We call this resolution {\it admissible} if: (i) each free $RG$-module $F_i$ in $\mathbf F$ comes with a given basis $\mathcal{X}_i$, and (ii) for each $x\in \mathcal{X}_i$, $\partial x
\neq 0 \in F_{i-1}$, while for  $x\in \mathcal{X}_0$, $\epsilon (x)=1\in R$, where 
$\epsilon :F_{0}\to R$ denotes the augmentation map. We write $\mathbf{F}^{(n)}$ for the $n$-{\it skeleton} ${\bigoplus^n_{i=0}}F_i$ which is free
with basis $\mathcal{X}^{(n)} = {\coprod_{i=0}}\mathcal{X}_i$. When $G$ has type $FP_n$ there is always an admissible free resolution with finitely generated $n$-skeleton.

If $\mathbf{F} \twoheadrightarrow R$ is admissible we think of
$\mathbf{F}$ as ${\bigoplus_{i\geq 0}}F_i$, a free module with basis
$\mathcal{X} = {\coprod_{i\geq 0}}\mathcal{X}_i$.  {\it The basic
valuation on} $\mathbf{F}$ {\it extending} $\chi$ is the function
$v_{\chi }: \mathbf{F} \to \BR \cup \{\infty\}$ of the kind described
above, where the values on $\mathcal{X}$ are chosen
inductively, skeleton by skeleton, to satisfy:

$$v_{\chi}(x) = \begin{cases} 0 &\text{if $i =0$}\cr\\
v_{\chi}(\partial x) &\text{if $i > 0$}\end{cases}$$

and $v_{\chi }(gx)=\chi (g)+v_{\chi }(x)$ when $g\in G$.
It follows that $v_{\chi}(c) \leq v_{\chi}(\partial c)$ for all $c\in \mathbf{F}$, and $v_{\chi}(c) = \infty$ if and only if $c=0$.

One shows easily that if $\mathbf{F}$ has finitely generated $n$-skeleton
then every valuation  $v:{\mathbf F}^{(n)}\to {\mathbf R}$ on the
$n$-skeleton extending $\chi $ is {\it dominated} by the basic valuation $v_{\chi }$
in the sense that there is a number $\mu \geq 0$ with $v(c)\geq v_{\chi
}(c)-\mu $ for all $c\in {\mathbf F}^{(n)}$.

We note that when $\chi =0$
then $v$ takes all non-zero elements of $\mathbf{F}$ to $0$.

\section{The Invariant $\Sigma^n(G;R)$}\label{Sigma}

In this section we recall another definition of  $\Sigma ^n(G;R)$. Let
$\mathbf{F} \twoheadrightarrow R$ be an admissible free resolution
with $\mathcal{X}^{(n)}$ finite.  We denote
by $\mathbf{\tilde F}$ the augmented (exact) chain complex $\mathbf{F}\to
R \to 0$; i.e. $\tilde {F_i} =F_i$ for $i\neq -1$, $\tilde {F}_{-1}
= R$, and we write $\epsilon :F_{0}\to R$ for the augmentation map.
All valuations $v$ on $\mathbf{F}$ are extended to $\mathbf{\tilde F}$
by the convention $v(r) =\infty$ for every $r\in R$.

If $v$ is a basic valuation extending the non-zero character $\chi$, the property $v(c)\leq v\partial (c)$ ensures that $\mathbf{\tilde F}$ carries an $\BR$-graded filtration by $R$-subcomplexes $\mathbf{\tilde F}_v^{[t,\infty]}$ where $t\in \BR$ and ${\tilde F}_{v,i}^{[t,\infty]} := \{c \in {\tilde F}_i \mid v(c) \geq t\}$.  In particular,  ${\tilde F}_{v,-1}^{[t,\infty]} = R$ for all $t$.

We say that $\mathbf{\tilde F}$ is {\it controlled} $(n-1)$-{\it
acyclic} (abbrev.  $CA^{n-1}$) {\it with respect to} $v$ if there exists $\lambda
\geq 0$ such that, for all $t$ and all $0\leq p \leq n-1$, the inclusion
induces the zero homomorphism $$H_p({\mathbf{\tilde F}}_v^{[t,\infty]})
\to H_p(\mathbf{\tilde F}^{[t-\lambda,\infty]}_v).$$ This condition should be considered vacuous when $n=0$; i.e. $\mathbf{\tilde F}$ is always $CA^{-1}$ with respect to $v$.  

\begin{remark}\label{3.1}  \rm{If the $CA^{n-1}$-condition holds for some $t$ then $\lambda$ can be chosen so that it holds, using this same $\lambda $, for every $t$}. This is because a non-zero character defines a cocompact action on $\mathbb R$.  
\end{remark}

The following is the content of \cite[Theorem~3.2]{BieriRenz}:

\begin{theorem}\label{3.3}  Let $\chi : G \to \BR$ be a non-zero character.  If
$\mathbf{F} \twoheadrightarrow R$ is an admissible free resolution with finitely generated
$n$-skeleton, and $v :{\mathbf{F}} \to \BR \cup \{\infty\}$ is a basic valuation extending $\chi$,
then $[\chi ]\in \Sigma ^{n}(G;R)$ if and only if $\mathbf{\tilde F}$ is $CA^{n-1}$ with respect to $v$.
\hfill$\square$
\end{theorem}

\section{Valuations on tensor products}\label{Val} 

We now consider the product $G\times H$ of two groups.  Let $\mathbf{F}\twoheadrightarrow R$ be an admissible free resolution of the $RG$-module $R$ with basis $\mathcal{X}$, and let 
$\mathbf{F}' \twoheadrightarrow R$ be an admissible free resolution of the $RH$-module $R$ with basis $\mathcal{X'}$.  Then, with respect to the product action, $\mathbf F\otimes_R\mathbf{F}' \twoheadrightarrow R$ is a free resolution, and 
$\{x\otimes x' \mid x\in \mathcal{X}, x' \in \mathcal{X'}\}$ is a basis
for $\mathbf{F}\otimes_R\mathbf F'$.  Since $\partial (x\otimes x') =
\partial x\otimes x' \pm x\otimes \partial x'$, simple considerations show that 
$\mathbf{F} \otimes_R \mathbf{F}'\twoheadrightarrow R$ is admissible.

Let $\chi :G\to \BR$ and $\chi' : H\to \BR$ be characters. Let
$v_\chi$ and $v'_{\chi '}$ be the basic valuations on $\mathbf{F}$ and
$\mathbf{F}'$ extending $\chi$ and $\chi '$ respectively.  Denote by $w$
the basic valuation $v_{(\chi ,\chi ')}: \mathbf{F}\otimes_R\mathbf{F}'
\to \BR \cup \{\infty\}$ extending $(\chi ,\chi ')$.

\begin{prop}\label{4.1} $w(c\otimes c')=v_{\chi}(c)+v'_{\chi '}(c)$
\end{prop}

\begin{proof} We abbreviate $v_{\chi}$ to $v$ and $v'_{\chi '}$ to $v
'$.  The proof is by induction on the degree $n$ of $w(c\otimes c')$.
The case  $n=0$ is easy, so we assume the Proposition is true for
$n-1$. Let $x\otimes x'$ have degree $n\geq 1$.

\begin{equation*}\begin{aligned}
 w(x \otimes x') &=  w(\partial (x \otimes x'))\\
 &= w(\partial x \otimes x'+x \otimes \partial x')\\
 &=\min \{ w(\partial x \otimes x'), w(x \otimes \partial x')\}\\
 &=\min \{v(\partial x) + v'(x'), v(x)+v'(\partial x')\} \text {\rm by induction, since the supports are disjoint}\\
 &= v(x) + v'(x')\\
\end{aligned}\end{equation*}

So $$w(gx\otimes hx')=v(gx)+v'(hx')$$

Since $R$ is a domain, supp$(c\otimes c')=$ supp$(c)\times $supp$(c')$.
Thus $$w(c\otimes c')=v_{\chi}(c)+v_{\chi '}(c)$$.
\end{proof}

\section{Sigma Invariants of Products}\label{SigmaProd} 

It is convenient to split the proof of Theorem \ref{Thm1} into two propositions.

\begin{prop}\label{extreme} $[(\chi ,0)]\in \Sigma ^{n}(G\times H;R)$
if and only if $[\chi ]\in \Sigma ^{n}(G;R)$ 
\end{prop}

\begin{proof} The ``if" part of the Proposition follows from Meinert's Inequality. We are to prove the ``only if" part.  

With notation as in Section \ref{Val} we assume that
$\mathbf{F}$ and $\mathbf{F}'$ have  finitely generated $n$-skeleta.
Then $\mathbf{F} \otimes_R \mathbf{F'}$ also has finitely generated
$n$-skeleton.  Let $v$ denote the basic valuation on $\mathbf{F}$
extending the character $\chi$ on $G$. By Proposition \ref{4.1}, $w$,
defined by $w(c\otimes c ')=v(c)$, is the basic valuation extending the
character $(\chi ,0)$ on $G\times H$.

The chain complex $\mathbf{F}$ is a retract of $\mathbf{F} \otimes_R
\mathbf{F'}$ as follows.  We may assume (for convenience) that $F '_0$ is
generated by a single generator $x'$ and that $x'$ is mapped by $\epsilon
'$ to $1\in R$. Define $i:\mathbf{F}\to \mathbf{F} \otimes_R \mathbf{F'}$
by $i(c)=c\otimes x'$. Define $p:\mathbf{F} \otimes_R \mathbf{F'}\to
\mathbf{F}$ by $p(c\otimes c')=\epsilon (c')c$ when $c'$ has degree 0, and
$p(c\otimes c')=0$ when $c'$ has degree $>0$.  One checks that $i$ and $p$
are chain maps and that $p\circ i$ is the identity map. The composition
$v\circ p$ is a valuation on $\mathbf{F} \otimes_R \mathbf{F'}$ and is
therefore dominated by the basic valuation $w$.

Let $z\in {\mathbf F}$ be a $k$-cycle where $k<n$. By assumption there is a
number $\lambda \geq 0$ and a $(k+1)$-chain $d\in \mathbf{F} \otimes_R
\mathbf{F'}$ with $\partial d=i(z)$ and $w(d)\geq w(i(z))-\lambda.$
Hence we have $$\partial p(d)=p(\partial d)=p(i(z))=z$$ and

\begin{equation*}\begin{aligned}
 v(p(d))&= (v\circ p)(d)\\ &\geq w(d)-\mu\\ &\geq w(i(z))-\lambda -\mu\\
 &=v(z)-\lambda -\mu.\\ \end{aligned}\end{equation*}

\end{proof}

In view of Meinert's Inequality, the new content of Theorem \ref{Thm1} is:

\begin{prop}\label{5.1}  Let $R$ be a field.  Then
$$ \Sigma^n(G\times H; R)^c \supseteq \bigcup^n_{p=0}
\Sigma^p(G;R)^c * \Sigma^{n-p} (H;R)^c$$
\end{prop}

\begin{proof} The cases $p=0$ and $p=n$ {\it(mutatis mutandis)} are covered
by Proposition \ref{extreme}, so we will assume $1\leq p\leq n-1$.
As before, $v$ denotes a basic valuation on $\mathbf{F}$ extending
the character $\chi$ on $G$, and $v'$ denotes a basic valuation on
$\mathbf{F'}$ extending the character $\chi '$ on $H$.  We are to show
that if  $[\chi ]\in \Sigma^{p}(G;R)^{c}$ and  $[\chi ']\in \Sigma^{n-p}
(H;R)^{c}$ then $[(\chi, \chi ')]\in \Sigma^n(G\times H; R)^c $.

We view $G\mathcal {X}\times H\mathcal {X'}$ as the $R$-basis of
$\mathbf{F} \otimes_R \mathbf{F'}$.  We denote by $p$ and $p'$ the
projections onto the two factors  $G\mathcal {X}$ and $H\mathcal {X'}$.
Given $u\in {\mathbb R}$, each chain $y\in \mathbf{F} \otimes_R
\mathbf{F'}$ has a unique decomposition $y=y_{\lambda }+y_{\rho
}$ where $$v(p(\text {\rm supp}(y_{\lambda})))<u\leq v(p(\text
{\rm supp}(y_{\rho}))).$$ Thus, $y_\rho$ is the  ``subchain''
obtained from $y$ by setting equal to zero the coefficients of
all basis elements $gx \otimes hx'$ such that $gx$ does not lie in
$\mathbf{F}_v^{[u,\infty]}$.  (We think of $\lambda$ and $\rho$ as
standing for ``left" and ``right" -- see Figure 2.) We observe: \begin{enumerate} \item
For all $y\in {\mathbf F} \otimes_R \mathbf{F'}$, $y_{\lambda}=0$ if and
only if $v(p(\text {\rm supp}(y)))\geq u$.  \item Thus, in particular,
$(c\otimes c')_{\lambda }$=0$ \text{ \rm if and only if } v(c)=0$.
\item If $d$ and $e$ in ${\mathbf F} \otimes_R \mathbf{F'}$ have
disjoint supports then $(d+e)_{\lambda }= d_{\lambda }+e_{\lambda }$.
\end{enumerate} Similarly, given $u'\in {\mathbb R}$, each chain $y$ has
a unique decomposition $y=y_{\beta }+y_{\tau }$ where $$v'(p'(\text {\rm
supp}(y_{\beta})))<u'\leq v'(p'(\text {\rm supp}(y_{\beta}))).$$ Thus,
$y_\tau$ is the  ``subchain'' obtained from $y$ by setting equal to zero
the coefficients of all basis elements $gx \otimes hx'$ such that $hx'$
does not lie in $\mathbf{F}_v'^{[u',\infty]}$.  (We think of $\beta$
and $\tau$ as standing for ``bottom" and ``top".) We call the (fixed)
numbers $u$ and $u'$ {\it splitters}; they must be specified before this
notation can be used.

Without loss of generality we may assume $\mathbf{\tilde F}$ is $CA^{p-2}$ with respect to $v$ and  $\mathbf{\tilde {F}'}$ is $CA^{n-p-2}$ with respect to $v'$. For a chosen $(p-1)$-cycle 
$z\in \mathbf{F}$ and a chosen $(n-p-1)$-cycle $z'\in \mathbf{F'}$ define 

$$\eta (z):=\text{ inf}\{v(z)-v(c)|\partial c=z\}$$ 

and

$$\eta '(z'):=\text{ inf}\{v'(z')-v'(c')|\partial c'=z'\}$$ 

The numbers $\eta (z)$ and $\eta '(z')$ can be made arbitrarily
large by suitable choice of $z$ and $z' $. Once $z$ and $z' $ have
been chosen, we choose positive numbers $\mu <\eta (z)$ and $\mu ' <\eta '(z')$. 
Then choose chains $c$ and $c'$ such that $\partial c=z$,
$\partial c'=z'$,  $v(z)-\mu -1<v(c)\leq v(z)-\mu$ and 
$v'(z')-\mu ' -1<v'(c')\leq v'(z')-\mu '$. See Figure 1. 

\begin{figure}[h]
        \begin{center}
        \scalebox{0.5}{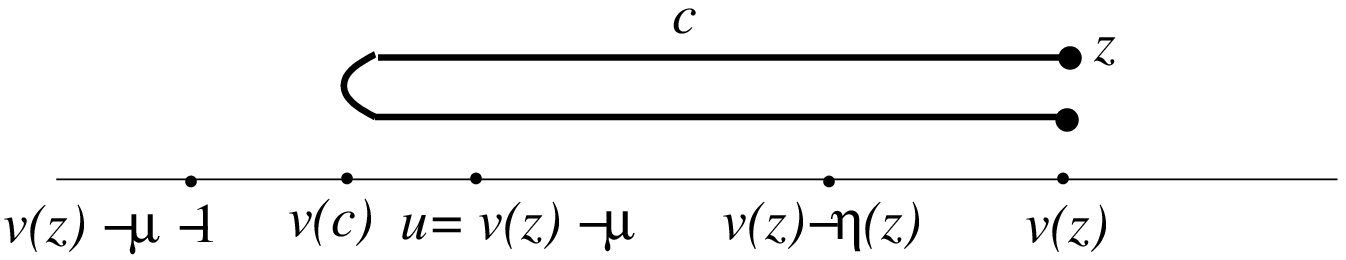}
        \caption{}
        \label{fig:Fig1}
        \end{center}
\end{figure}

Since $c\otimes c'$ is an $n$-chain, and $\partial (c\otimes c')=z\otimes c'\pm c\otimes z'$ we have:

\begin{equation*}
\begin{aligned}
w(\partial (c\otimes c'))&=w(z\otimes c'\pm c\otimes z')\\ 
&\geq \text{min}\{w(z\otimes c'), w(c\otimes z')\}\\
&=\text{min}\{v(z)+v'(c'), v(c)+v'(z')\}\\
&>\text{min}\{v(z)+v'(z')-\mu '-1, v(z)+v'(z')-\mu -1\}\\
&=w(z\otimes z')-1-\text{ max}\{\mu,\mu '\}
\end{aligned}
\end{equation*}

Now we take as our splitters $u:=v(z)-\mu$ and $u ':=v'(z')-\mu
'$.  We consider an arbitrary $n$-chain $d$ such that $\partial
d=\partial(c\otimes c')$.  Define 
$$b:=\partial (d_{\lambda})-(\partial d)_{\lambda}$$ 
and define $e:=\partial (b_{\beta})$.  Thus $e$ is
an $(n-2)$-cycle. See Figure 2. 
\begin{figure}[h]
        \begin{center}
        \scalebox{0.5}{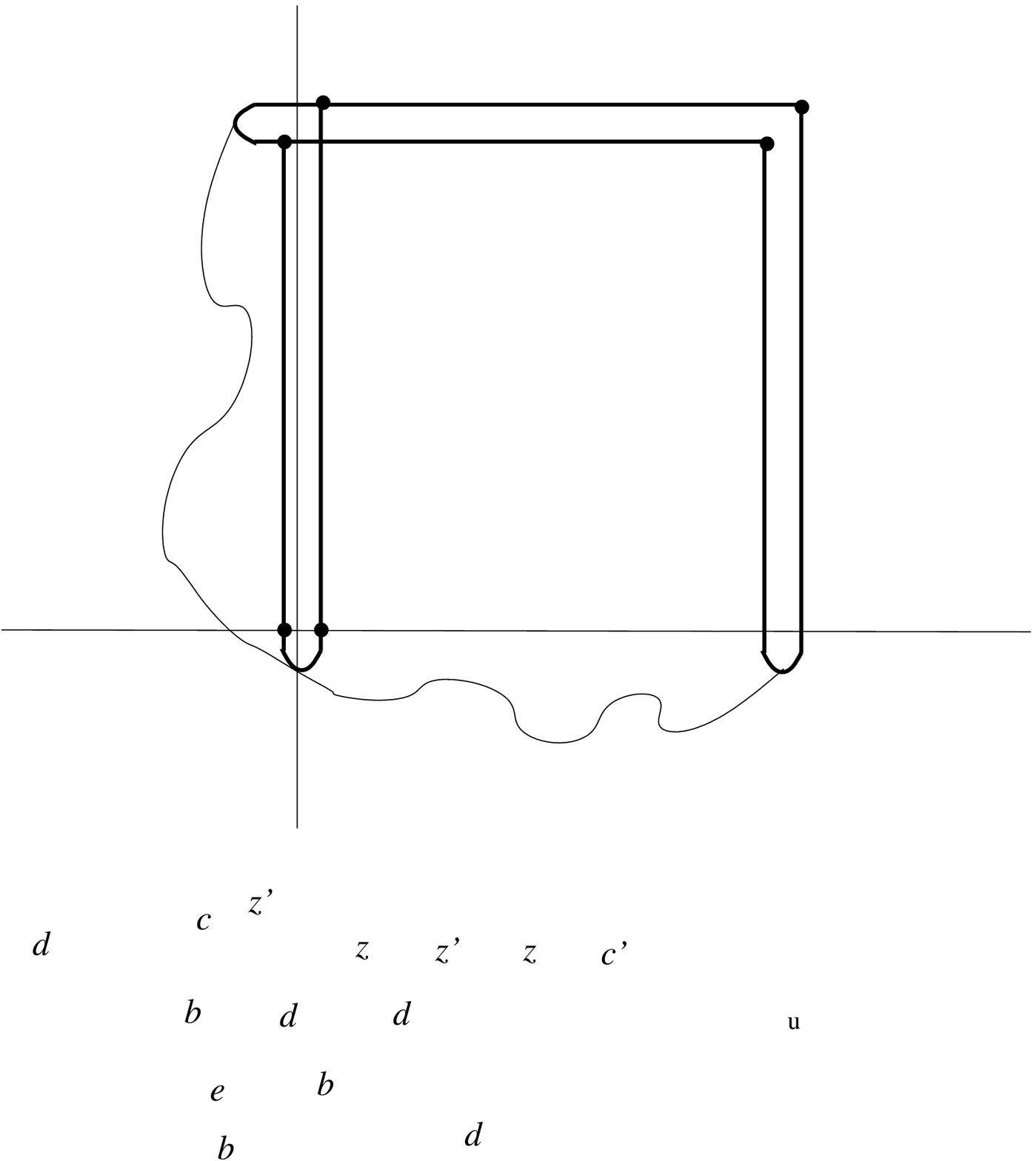}
        \caption{}
        \label{fig:Fig2}
        \end{center}
\end{figure}

{\it Claim 1:} $(\partial d)_{\lambda}=\pm (c\otimes z')_{\lambda}$.

{\it Proof:} $\partial d=\partial (c\otimes c')=z\otimes c'\pm c\otimes z'$. 
Since $v(p(\text {\rm supp}(z\otimes c')))\geq v(\text {\rm supp}(z))>u$ we have 
$(z\otimes c')_{\lambda}=0$.  Thus 

\begin{equation*}
\begin{aligned}
(\partial d)_{\lambda}&=(z\otimes c')_{\lambda }\pm (c\otimes z')_{\lambda }\\
&=\pm (c\otimes z')_{\lambda }\\
\end{aligned}
\end{equation*}
as claimed.

Claim 1 is used in:

\begin{equation*}\begin{aligned}
z\otimes z'&=\partial (c\otimes z')\\
&=\partial ((c\otimes z')_{\lambda})+\partial ((c\otimes z')_{\rho})\\
&=\pm \partial ((\partial d)_{\lambda})+\partial ((c\otimes z')_{\rho})\\
&=\pm \partial (b)+\partial ((c\otimes z')_{\rho})\\
&=\pm \partial (b_{\beta})\pm \partial (b_{\tau})+\partial ((c\otimes z')_{\rho})\\
&=\pm e\pm \partial (b_{\tau})+\partial ((c\otimes z')_{\rho})
\end{aligned}\end{equation*}

Next, we show that the indicated homology between $z\otimes z'$ and $e$ takes place in the chain complex

$$\mathbf{C}:=\mathbf{F}_v^{[u,\infty]}\otimes\mathbf{F'}_{v '}^{[u',\infty]}$$

{\it Claim 2:} $b=(\partial (d_{\lambda}))_{\rho }$.

{\it Proof:} $d=d_{\lambda}+d_{\rho}$. So $\partial d=\partial (d_{\lambda})+\partial (d_{\rho})$.  Thus $(\partial d)_{\lambda}=(\partial (d_{\lambda}))_{\lambda}$, because $\mathbf{F}_v^{[u,\infty]}$ is a chain complex, so $(\partial (d_{\rho}))_{\lambda }=0.$
Since $b=(\partial (d_{\lambda}))_{\rho} + (\partial (d_{\lambda}))_{\lambda} -(\partial d)_{\lambda}$ the Claim follows.

{\it Claim 3:} $c\otimes z'=(c\otimes z')_{\tau}$.

{\it Proof:} This is because $v'p'(\text {\rm supp}(c\otimes z'))>u'$.

It follows from Claims 2 and 3 that $b_{\tau}=(b_{\rho})_{\tau}$ and $(c\otimes z')_{\rho}=((c\otimes z')_{\tau})_{\rho}$, so $e$ and $z\otimes z'$ are homologous in $\mathbf C$.

We now use the fact that $z$ does not bound in $\mathbf{\tilde
F}_{v}^{[u,\infty]}$ and $z'$ does not bound in $\mathbf{\tilde
{F'}}_{v'}^{[u',\infty]}$.  Because $R$ is a field, the K\"{u}nneth
Formula (\cite[Lemma 5.3.1]{Spanier}) applied to $\mathbf C$ implies
that the homology class of $z\otimes z'$ in $\mathbf{C}$ could only
be zero if the homology class of either $z$ or $z'$ is zero in the
appropriate tensor factor of $\mathbf{C}$; and neither is zero, as
we have just seen.  Thus the homology class of  $e$ is non-zero, and
hence the cycle $e=\partial (b_{\beta })$ is non-zero.  It follows that
$b_{\beta }$ is non-zero. Then Claim 2 implies that $(\partial
(d_{\lambda})_{\rho })_{\beta}$ is non-zero.

{\it Claim 4:} $(d_{\lambda})_{\beta }\neq 0$.

{\it Proof:} Since the support of  $(\partial (d_{\lambda})_{\rho
})_{\beta}$ is non-zero, the support of $(\partial (d_{\lambda})_{\rho })$
contains some $gx\otimes hx'$ with $v'(hx')<u'$.  So the same is true of
$\partial (d_{\lambda})$.  Hence the support of $d_{\lambda}$ contains
some $\tilde {g}\tilde {x}\otimes \tilde {h}\tilde {x'}$ with $v'(\tilde
{h}\tilde {x'})<u'$. Thus $(d_{\lambda})_{\beta }\neq 0$ as claimed.

By Claim 4, we get: 

\begin{equation*}\begin{aligned} 
w(d)&\leq w(\tilde {g}\tilde {x}\otimes \tilde {h}\tilde {x'})\\
&=v(\tilde {g}\tilde {x})+v '(\tilde {h}\tilde {x'})\\
&\leq v(z)-\mu + v '(z ')-\mu '\\
&=w(z\otimes z ')-\mu -\mu'
\end{aligned}\end{equation*}

Summarizing: since $\partial d=\partial (c\otimes c')$ we conclude that 

$$w(\partial d)-w(d)\geq w(z\otimes z')-1-\text{ max}\{\mu ,\mu '\}
-w(z\otimes z ')+\mu +\mu '$$

i.e.

$$w(\partial d)-w(d) \geq \text { min}\{\mu ,\mu '\}$$

Since $\mu$ and $\mu '$ can be chosen arbitrarily large, say $>N$, the cycle $\partial (c\otimes c')$ has the property that for
any $d$ with $\partial d=\partial (c\otimes c')$, $w(\partial d)-w(d)>N-1$. In other words, $$\widetilde{\mathbf{F} \otimes_R \mathbf{F'}} \text{ is not } CA^{n-1} \text{ with
respect to } w$$
\end{proof}

\begin{remark}\label{rational} \rm{The only place where we needed $R$ to
be a field was to ensure that (referring to homology classes) $\{z\otimes
z'\}=0$ forces $\{z\}=0$ or $\{z'\}$=0.  This also holds when $R=\BZ$
provided $\{z\}$ and $\{z'\}$ have infinite order.  Thus our proof also
gives Theorem \ref{Thm2} because under the hypotheses of Theorem \ref{Thm2}
the $\BZ$-cycles $z$ and $z'$ can always be chosen so that their homology
classes have infinite order. (This is just a variant on the proof given in Section \ref{Intro}.) We also get Theorem \ref{Thm3} because in
the cases where both sides of the join are non-empty the dimensions of the
relevant cycles are 0 and 1, and the 0-cycle can always be chosen to
be indivisible.}
\end{remark}

\bibliographystyle{amsplain}

\end{document}

%% file: pic1.pstex_t
\begin{picture}(0,0)%
\includegraphics{pic1.pstex}%
\end{picture}%
\setlength{\unitlength}{3947sp}%
\begingroup\makeatletter\ifx\SetFigFont\undefined%
\gdef\SetFigFont#1#2#3#4#5{%
  \reset@font\fontsize{#1}{#2pt}%
  \fontfamily{#3}\fontseries{#4}\fontshape{#5}%
  \selectfont}%
\fi\endgroup%
\begin{picture}(6372,1218)(2011,-1339)
\put(6826,-1261){\makebox(0,0)[lb]{\smash{{\SetFigFont{14}{16.8}{\rmdefault}{\mddefault}{\updefault}{\color[rgb]{0,0,0}$v(z)$}%
}}}}
\put(5476,-1261){\makebox(0,0)[lb]{\smash{{\SetFigFont{14}{16.8}{\rmdefault}{\mddefault}{\updefault}{\color[rgb]{0,0,0}$v(z)-\eta (z)$}%
}}}}
\put(3901,-1261){\makebox(0,0)[lb]{\smash{{\SetFigFont{14}{16.8}{\rmdefault}{\mddefault}{\updefault}{\color[rgb]{0,0,0}$u=v(z)-\mu$}%
}}}}
\put(2026,-1261){\makebox(0,0)[lb]{\smash{{\SetFigFont{14}{16.8}{\rmdefault}{\mddefault}{\updefault}{\color[rgb]{0,0,0}$v(z)-\mu -1$}%
}}}}
\put(3376,-1261){\makebox(0,0)[lb]{\smash{{\SetFigFont{14}{16.8}{\rmdefault}{\mddefault}{\updefault}{\color[rgb]{0,0,0}$v(c)$}%
}}}}
\end{picture}%

%% file: pic2.pstex_t
\begin{picture}(0,0)%
\includegraphics{pic2.pstex}%
\end{picture}%
\setlength{\unitlength}{3947sp}%
\begingroup\makeatletter\ifx\SetFigFont\undefined%
\gdef\SetFigFont#1#2#3#4#5{%
  \reset@font\fontsize{#1}{#2pt}%
  \fontfamily{#3}\fontseries{#4}\fontshape{#5}%
  \selectfont}%
\fi\endgroup%
\begin{picture}(7959,6369)(1654,-6103)
\put(7951,-436){\makebox(0,0)[lb]{\smash{{\SetFigFont{14}{16.8}{\rmdefault}{\mddefault}{\updefault}{\color[rgb]{0,0,0}$z\otimes z'$}%
}}}}
\put(4201,-4486){\makebox(0,0)[lb]{\smash{{\SetFigFont{14}{16.8}{\rmdefault}{\mddefault}{\updefault}{\color[rgb]{0,0,0}$e=\partial b_{\beta}$}%
}}}}
\put(4126,-4861){\makebox(0,0)[lb]{\smash{{\SetFigFont{14}{16.8}{\rmdefault}{\mddefault}{\updefault}{\color[rgb]{0,0,0}$b_{\beta}$}%
}}}}
\put(5776,-5686){\makebox(0,0)[lb]{\smash{{\SetFigFont{14}{16.8}{\rmdefault}{\mddefault}{\updefault}{\color[rgb]{0,0,0}$d$}%
}}}}
\put(5401,-361){\makebox(0,0)[lb]{\smash{{\SetFigFont{14}{16.8}{\rmdefault}{\mddefault}{\updefault}{\color[rgb]{0,0,0}$c\otimes z'$}%
}}}}
\put(7951,-2461){\makebox(0,0)[lb]{\smash{{\SetFigFont{14}{16.8}{\rmdefault}{\mddefault}{\updefault}{\color[rgb]{0,0,0}$c'\otimes z$}%
}}}}
\put(2851,-2086){\makebox(0,0)[lb]{\smash{{\SetFigFont{14}{16.8}{\rmdefault}{\mddefault}{\updefault}{\color[rgb]{0,0,0}$d_{\lambda}$}%
}}}}
\put(4201,-2836){\makebox(0,0)[lb]{\smash{{\SetFigFont{14}{16.8}{\rmdefault}{\mddefault}{\updefault}{\color[rgb]{0,0,0}$b=\partial (d_{\lambda})-(\partial d)_{\lambda }$}%
}}}}
\put(4036,-5986){\makebox(0,0)[lb]{\smash{{\SetFigFont{14}{16.8}{\rmdefault}{\mddefault}{\updefault}{\color[rgb]{0,0,0}right}%
}}}}
\put(8896,-4441){\makebox(0,0)[lb]{\smash{{\SetFigFont{14}{16.8}{\rmdefault}{\mddefault}{\updefault}{\color[rgb]{0,0,0}top}%
}}}}
\put(3556,-6001){\makebox(0,0)[lb]{\smash{{\SetFigFont{14}{16.8}{\rmdefault}{\mddefault}{\updefault}{\color[rgb]{0,0,0}left}%
}}}}
\put(8896,-4816){\makebox(0,0)[lb]{\smash{{\SetFigFont{14}{16.8}{\rmdefault}{\mddefault}{\updefault}{\color[rgb]{0,0,0}bottom}%
}}}}
\end{picture}%

%% file: product21.bbl
\begin{thebibliography}{99}

\bibitem{BB} Bestvina, Mladen ;  Brady, Noel . Morse theory and finiteness properties of groups. Invent. Math.  129  (1997),  no. 3, 445

\bibitem{Bieri} Bieri, Robert . Finiteness length and connectivity length for groups.
 Geometric group theory down under (Canberra, 1996), 9--22, de Gruyter, Berlin,  1999.

\bibitem{RossBieri} Bieri, Robert ;  Geoghegan, Ross . Connectivity properties of group actions on non-positively curved spaces. Mem. Amer. Math. Soc.  161  (2003),  no. 765, xiv+83 pp.

\bibitem{RossBieri3} Bieri, Robert; Geoghegan, Ross . (paper in prepartion)

\bibitem{BGK} Bieri, Robert;  Geoghegan, Ross;  Kochloukova, Dessislava . The Sigma invariants of Thompson's Group $F$, preprint.

\bibitem{BNS} Bieri, Robert ;  Neumann, Walter D. ;  Strebel, Ralph . A geometric invariant of discrete groups.  Invent. Math.  90  (1987),  no. 3, 451--477.

\bibitem{BieriRenz} Bieri, Robert ;  Renz, Burkhardt . Valuations on free resolutions and higher geometric invariants of groups. Comment. Math. Helv.  63  (1988),  no. 3, 464.497.

\bibitem{BuxG} Bux, Kai-Uwe ;  Gonzalez, Carlos . The Bestvina-Brady construction revisited: geometric computation of  $\Sigma$-invariants for right-angled Artin groups.
 J. London Math. Soc. (2)  60  (1999),  no. 3, 793--801.

\bibitem{Gehrke} Gehrke, Ralf . A lower bound for the $\Sigma ^2$ complement of groups with sufficient commutativity, Thesis, University of Frankfurt,

\bibitem{Gehrke2} Gehrke, Ralf . The higher geometric invariants for groups with sufficient
 commutativity. Comm. Algebra  26  (1998),  no. 4, 1097 

\bibitem{Rossbook} Geoghegan, Ross . Topological methods in group theory.
Graduate Texts in Mathematics, 243. Springer, New York,  2008. xiv+473 pp. ISBN: 978-0-387-74611-1

\bibitem{MMVW} Meier, John ;  Meinert, Holger ;  VanWyk, Leonard . On the
$\Sigma$-invariants of Artin groups. Geometric topology and geometric
group theory (Milwaukee, WI,  1997).  Topology Appl.  110  (2001),
no. 1, 71--81.

\bibitem{MMVW2} Meier, John ;  Meinert, Holger ;  VanWyk, Leonard . On the
$\Sigma$-invariants of graph products based on trees.  Preprint.

\bibitem{Meinert} Meinert, Holger . Diplome Thesis, University of Frankfurt, 1990

\bibitem{Dirk} Sch\"{u}tz, Dirk . On the direct product conjecture for sigma invariants, preprint, University of Durham.

\bibitem{Spanier} Spanier, Edwin H.  Algebraic topology. Corrected reprint.
Springer-Verlag, New York-Berlin,  1981. xvi+528 pp. ISBN: 0-387-90646-0

\end{thebibliography}
